\documentclass[12pt,a4paper]{article}
\usepackage{fancyhdr}
\usepackage{amsmath}
\usepackage[latin1]{inputenc}
\usepackage{graphicx}
\usepackage{url}

\usepackage{amssymb}
\usepackage{amsthm}

\newcommand{\N}{{\mathbb N}}
\newcommand{\R}{{\mathbb R}}
\newcommand{\Z}{{\mathbb Z}}
\newcommand{\Q}[1]{Q(#1)}
\newcommand{\Qi}[1]{Q^{*}(#1)}
\newcommand{\Pint}[1]{P^{*}(#1)}
\newcommand{\Qp}[1]{Q'(#1)}
\newcommand{\C}[2]{C_{#1}^{#2}}
\newcommand{\ceil}[1]{\left\lceil #1 \right\rceil}
\newcommand{\floor}[1]{\left\lfloor #1 \right\rfloor}
\newcommand{\card}[1]{\left| #1 \right|}
\newcommand{\sys}{\mathcal{S}}

\newcommand{\conv}{\mathrm{conv}}
\newcommand{\onevec}{\mathbf{1}}
\newcommand{\zerovec}{\mathbf{0}}
\newcommand{\OM}{\mathbb{O}}
\newcommand{\tabulatedset}[1]%
 {%
 \left\{ #1 \right\}%
 }
\newcommand{\setof}[2]%
 {%
 \left\{ #1 \, : \, #2 \right\}%
 }

\newtheorem{theorem}{Theorem}[section]
\newtheorem{lemma}[theorem]{Lemma}
\newtheorem{corollary}[theorem]{Corollary}

\newtheorem{remark}[theorem]{Remark}

\begin{document}

\begin{verbatim}\end{verbatim}\vspace{2.5cm}

\begin{center}
\textsc{\Large\textbf{Generalized minor inequalities for the set covering polyhedron
related to circulant matrices}}
\end{center}
\smallskip

\begin{center}
\large Paola B.\,Tolomei\footnote{ptolomei@fceia.unr.edu.ar}
\end{center}
\begin{center}
\textit{Departamento de Matem\'atica, Facultad de Ciencias Exactas, Ingenier\'ia y Agrimensura, Universidad Nacional de Rosario, Santa Fe, Argentina}
\medskip

\textit{and CONICET, Argentina}
\end{center}

\smallskip

\begin{center}
\large Luis M.\,Torres\footnote{luis.torres@epn.edu.ec}
\end{center}

\begin{center}
\textit{Centro de Modelización Matemática ModeMat - Escuela Politécnica Nacional, Quito, Ecuador}
\end{center}

\medskip

\begin{abstract}
We study the set covering polyhedron related to circulant matrices.
In particular, our goal is to characterize the first Chvátal closure
of the usual fractional relaxation. We present a family of
valid inequalities that generalizes the family of minor inequalities
previously reported in the literature and includes
new facet-defining inequalities. 
Furthermore, we propose a polynomial time separation algorithm for a particular subfamily of 
these inequalities.

\medskip

\end{abstract}

\section{Introduction}
The \emph{weighted set covering problem} can be stated as 
$$
\mbox{(SCP)} \;\;\min\{c^Tx: Ax\geq \onevec, x\in  \{0,1\}^n\}
$$
where $A$ is an $m\times n$ matrix with $0,1$ entries, $c \in \Z^n$,
and $\onevec \in \R^m$ is the vector having all entries equal to one.
The SCP is a classic problem in combinatorial optimization 
with important practical applications (crew scheduling, facility location,
vehicle routing, to cite a few prominent examples), but hard to
solve in general. One
established approach to tackle this problem is to study the
polyhedral
properties of the 
set of its feasible solutions. \cite{BalasNg89,CornuejolsSassano89,NobiliSassano89,Sa}.

The \emph{set covering polyhedron} $\Qi{A}$ is defined as the convex
hull of all feasible solutions of SCP. Its \emph{fractional
  relaxation} $\Q{A}$ is the feasible region of the linear programming
relaxation of SCP, i.e., 
$$
\Q{A}:=\{x\in [0,1]^n: Ax\geq \onevec \}.
$$

It is known that SCP can be solved in polynomial time if $A$ belongs
to the particular class of \textit{circulant matrices} defined in the
next section. Hence, it is natural to ask whether an
explicit description in terms of linear inequalities can be provided for $\Qi{A}$ in
this case, an issue that has been addressed in several recent studies by researchers
in the field (see \cite{circu,BNTDAM13,BNTMMOR13,CN} among others). For the related
\emph{set packing polytope} of circulant matrices,
$$
\Pint{A}:= \conv (\{Ax\leq \onevec, x\in  \{0,1\}^n\})
$$
such a description follows from the results published in \cite{EisenbrandEtAl05}.

Bianchi et al.~introduced in \cite{BNTDAM13} a family of facet-defining
inequalities for $\Qi{A}$ which are associated with certain structures
called circulant minors. Moreover, the authors presented in \cite{BNTMMOR13} two families
of circulant matrices for which $\Qi{A}$ is completely described by
this class of \emph{minor inequalities}, together with the full-rank
inequality and the inequalities defining $\Q{A}$, usually denoted as
boolean facets. The existence of a third family of circulant matrices having this
property follows from previous results obtained by Bouchakour et al.~\cite{Mah} in the context
of the dominating set polytope of some graph classes.

If an inequality $a^Tx \leq b$ is valid for a polytope $P \subset \R^n$ and
$a \in \Z^n$, then $a^Tx \leq \floor{b}$ is valid for the integer polytope 
$P_I:= \conv(P \cap \Z^n)$. This procedure is called Chvátal-Gomory 
rounding, and it is known that the system of all linear inequalities which 
can be obtained in this way defines a new polytope $P'$, the 
\emph{first Chvátal closure of $P$}. Moreover, iterating this procedure yields $P_I$
in a finite number of steps. An inequality is said to have \emph{Chvátal rank} of 
$t$ if it is valid for the $t$-th Chvátal closure of a polytope.
All inequalities mentioned above have Chvátal rank less than or
equal to one.

With the aim of investigating if the results in \cite{BNTMMOR13,Mah}
can be generalized to all circulant matrices, 
we have tried to characterize the first Chvátal
closure of $\Q{A}$ for any circulant matrix $A$.
In particular, we addressed the question whether the system consisting 
of minor inequalities, boolean facets, and the
full-rank inequality is sufficient for describing $\Qp{A}$. We have obtained a
negative answer to this question in the form of a new class of valid inequalities for $\Qi{A}$ which contains minor
inequalities as a proper subclass. All inequalities from this class
have Chvátal rank equal to one, and besides, some of them define new
facets of $\Qi{A}$, as we show by an example.

This paper is organized as follows. In the next section, we introduce some
notation and preliminary results required for our work. In Section~\ref{sec:clausura}
we describe our approach for computing the first Chvátal closure of $\Q{A}$
and define the new class of \emph{generalized minor inequalities}. A separation 
algorithm for a particular subclass of these is provided in Section \ref{sec:separation}.
Finally, some conclusions and possible directions for future work are discussed
in Section~\ref{sec:conclusions}. A preliminary version of this article appeared without proofs in \cite{TolomeiTorres13}.

\section{Notations, definitions and preliminary results}
\label{sec:notations}

For $n \in \N$, let $[n]$ denote the additive group defined on the set
$\tabulatedset{1, \ldots, n}$, with integer addition modulo $n$.
Throughout this article, if $A$ is a $0,1$ matrix of order $m\times
n$, then we consider the columns (resp. rows) of $A$ to be indexed by
$[n]$ (resp. by $[m]$). In particular, addition of column (resp. row) indices
is always considered to be taken modulo $n$ (resp. modulo $m$). Two matrices $A$ and $A'$ are
\emph{isomorphic}, denoted by $A\approx A'$, if $A'$ can be
obtained from $A$ by permutation of rows and columns. Moreover, we say
that a row $v$ of $A$ is a \textit{dominating row} if $v\geq u$ for
some other row $u$ of $A$, $u\neq v$.

Given $N\subset [n]$, the \emph{minor of} $A$ \emph{obtained by
  contraction of} $N$, denoted by $A/N$, is the submatrix of $A$ that
results after removing all columns with indices in $N$ and all
dominating rows.
In this work, when we refer to \emph{a minor} of $A$ we always
consider a minor obtained by contraction.

Let $n, k \in \N$ with $2\leq k\leq n-2$, and
$C^i:=\{i,i+1,\ldots,i+(k-1)\}\subset [n]$ for every $i\in [n]$.  
With a little abuse of notation we will also use $C^i$ to denote
the incidence vector of this set. A
\emph{circulant matrix} $C_n^{k}$ is the square matrix of order $n$ whose $i$-th
row vector is $C^i$.  Observe that 
$C^i=\sum_{j=i}^{i + k-1} e^j$, where $e^j$ is the $j$-th canonical vector
in $\R^n$. 

A minor of $C_n^k$ is called a \emph{circulant minor} if it is isomorphic to 
a circulant matrix $C_{n'}^{k'}$. As far as we are aware, circulant minors were 
introduced for the first time in \cite{CN}, where the authors used them as a tool
for establishing a complete description of ideal and minimally nonideal circulant matrices. 
More recently, Aguilera \cite{Nes} completely characterized the subsets $N$ of
$[n]$ for which $\C{n}{k}/N$ is a circulant minor. We review at next his
main result, as some terms will be needed for the separation algorithm presented
in Section~\ref{sec:separation}. 

Given $C^{k}_n$, the digraph $G(C^{k}_n)$ has vertex set $[n]$ and $(i,j)$ is an arc 
of $G(C^{k}_n)$ if $j\in \{i+k,i+k+1\}$. We call arcs of the form $(i,i+k)$ \emph{short} arcs
and arcs of the form $(i,i+k+1)$ \emph{long} arcs. Associated with any directed cycle $D$
in $G(C^{k}_n)$, three parameters can be defined: its number $n_2$ of short arcs, its number $n_3$
of long arcs, and the number $n_1$ of turns around the set of nodes  it makes. Hence, the
relationship $n_1 n = n_2 k + n_3(k+1)$ must hold for the integers $n, k, n_1, n_2,$ and $n_3$.
In our current notation, Theorem 3.10 of \cite{Nes} states the following.

\begin{theorem} \cite{Nes} \label{minorsgrafo}
Let $n,k$ be integers verifying $2\leq k \leq n-1$ and let $N\subset \Z_n$ such that $1\leq |N|\leq n-2$.
Then, the following are equivalent:
\begin{enumerate}
\item $C_n^k/N$ is isomorphic to $C_{n'}^{k'}$.  
\item $N$ induces in $G(C^{k}_n)$ $d\geq 1$ disjoint simple dicycles $D_0,\ldots,D_{d-1}$, each of them having the same parameters $n_1$, $n_2$ and $n_3$ and such that 
$|N|=d(n_2+n_3)$, 
$n'=n-d(n_2+n_3)\geq 1$, and $k'=k-dn_1\geq 1$.
\end{enumerate}
\end{theorem}

The structure of $\Qi{\C{n}{k}}$ has been the subject of many previous
studies. It is known that $\Qi{\C{n}{k}}$ is a full dimensional
polyhedron. Furthermore, for every $i\in [n]$, the constraints
$x_i\geq 0$, $x_i\leq 1$ and $\sum_{j\in C^{i}} x_j \geq 1$ are facet
defining inequalities of $\Qi{\C{n}{k}}$ and we call them
\emph{boolean facets} \cite{Sa}. We will denote by $\sys_0$ the system of linear
inequalities corresponding to boolean facets.

The \emph{rank constraint} $\sum_{i=1}^n x_i\geq \ceil{\frac{n}{k}}$
is always valid for $\Qi{\C{n}{k}}$ and defines a facet if and only if $n$
is not a multiple of $k$ \cite{Sa}.
In \cite{BNTDAM13} the authors obtained another family of
facet-defining inequalities for $\Qi{\C{n}{k}}$ associated with
circulant minors.

\begin{lemma} \cite{BNTDAM13} \label{nosotras}
Let $N\subset [n]$ such that $\C{n}{k}/ N \approx \C{n'}{k'}$,
and let $W=\{i\in N : i-k-1 \in N\}$.
Then, the  inequality 
\begin{equation}\label{ecu}
\sum_{i\in W} 2 x_i + \sum_{i\notin W} x_i \geq \ceil{\frac{n'}{k'}}
\end{equation}
is a valid inequality for $Q'(C^k_n)$. Moreover, if $2\leq k'\leq
n'-2$, $\left\lceil \frac{n'}{k'}\right\rceil >
\left\lceil\frac{n}{k}\right\rceil$ and $n'= 1(\mathrm{mod} \,k')$,
this inequality defines a facet of $\Qi{\C{n}{k}}$.
\end{lemma}

Observe that the set $W$ uniquely determines the set $N$, and hence
the minor $\C{n}{k}/ N$ associated with it.
The authors termed \eqref{ecu} as the \emph{minor inequality
  corresponding to} $W$. Moreover, the authors showed that
every non boolean facet-defining inequality of $\Qi{C^3_n}$ (whose
facetial structure had been previously characterized in \cite{Mah}) 
is either the rank constraint or a minor inequality. 
Similarly, $\Qi{C_{2k}^k}$ and $\Qi{C_{3k}^k}$ are completely described by boolean facets and 
minor inequalities, for any $k \geq 2$ \cite{BNTMMOR13}.

\section{Computing the first Chvátal closure}
\label{sec:clausura}

In our attempt at finding a linear description of the first
Chvátal closure of $\Q{\C{n}{k}}$, we use the following well-known
result from integer programming:
\begin{lemma}
Let $P=\setof{x \in \R^n}{Ax \geq b}$ be a nonempty polyhedron
with $A$ integral and $Ax \geq b$ totally dual integral. Then,
$P'=\setof{x \in \R^n}{Ax \geq \ceil{b}}$.
\end{lemma}

As we shall see below, if all vertices of $P$ are known then a totally
dual integral system describing the polyhedron can be computed from
this information. This is the case for $\Q{\C{n}{k}}$, whose vertices
have been completely characterized by Argiroffo and Bianchi
\cite{circu}. 
\begin{lemma}[\cite{circu}]
\label{th:vertices}
Let $x^*$ be a vertex of $\Q{\C{n}{k}}$. 
Then one of the following statements holds:
\begin{enumerate}
\renewcommand{\theenumi}{\roman{enumi}}
\renewcommand{\labelenumi}{(\theenumi)}
\item $x^*$ is integral.
\item $x^*= \frac{1}{k} \onevec$
\item There exists $N \subset [n]$ with $\C{n}{k}/N \cong
\C{n'}{k'}$ and $\gcd(n', k')=1$ such that
$$
x_i^*:= \left\{
\begin{array}{ll}
\frac{1}{k'}, & \mbox{if $i \not\in N$}, \\ 
0, & \mbox{otherwise}.
\end{array}
\right.
$$
\end{enumerate}
\end{lemma}

In order to express $\Q{\C{n}{k}}$ via a totally dual integral system
of linear inequalities, we use the method described below
(see, e.g., \cite[Ch.~8]{libro}). Given a polyhedral cone $K \subset \R^n$,
consider the points in the lattice $L:= K \cap \Z^n$. An
\emph{integral generating set} for $L$ is a set $H \subseteq L$ having
the property that every $x \in L$ can be written as a linear
combination $\sum_{i=1}^k \alpha_i h_i$ of some elements $h_1, \ldots,
h_k \in H$ with integral non negative coefficients $\alpha_1, \ldots,
\alpha_k$.

The method consists in adding redundant inequalities to the
original system $\sys_0$ until the following property is verified: If
$\setof{a_i^T x \geq b_i}{i \in I(x^*)}$ is the set of linear
inequalities satisfied with equality by a vertex $x^* \in
\Q{\C{n}{k}}$ and $K(x^*)$ is the cone generated by the set of vectors 
$H(x^*):= \setof{a_i}{i \in I(x^*)}$, then $H(x^*)$ is
an integral generating set for $K(x^*) \cap \Z^n.$

This idea leads to the following procedure for computing $\Qp{\C{n}{k}}$:
\begin{enumerate}
\renewcommand{\theenumi}{\arabic{enumi}}
\renewcommand{\labelenumi}{\theenumi.}
\renewcommand{\theenumii}{\arabic{enumii}}
\renewcommand{\labelenumii}{\theenumi.\theenumii}
\item Let $\sys:= \sys_0$. 
\item For every vertex $x^*$ of $\Q{\C{n}{k}}$ do
\begin{enumerate}
\item Compute an integral generating set $H(x^*)$ of
$K(x^*) \cap \Z^n$.
\item \label{proc:add-ineq}For all $a_i \in H(x^*)$, let $b_i:= a_i^T x^*$
and add the inequality $a_i^T x \geq \ceil{b_i}$ to $\sys$.
\end{enumerate}
\item Return $\sys$ as a linear description of $\Qp{\C{n}{k}}$.
\end{enumerate}

Observe that the inequality $a_i^T x \geq b_i$ at step~2.2 is valid
for $\Q{\C{n}{k}}$, since $x^*$ minimizes $a_i^T x$ over this polyhedron
for any $a_i \in K(x^*)$. Moreover, if $x^*$ is integral, then the new
inequality added to the system $\sys$ is redundant, as $b_i \in
\Z$. Therefore, new inequalities for $\Qp{\C{n}{k}}$ may only arise
from integer generating sets $H(x^*)$ related to fractional vertices belonging to
one of the two latter clases described in Lemma~\ref{th:vertices}.

Firstly, we analyze all inequalities arising from the vertex
$x^*=\frac{1}{k} \onevec$.  The point $x^*=\frac{1}{k} \onevec$ is
known to be a vertex of $\Q{\C{n}{k}}$ if and only if
$\gcd(n,k)=1$. In this case, $\setof{C^i x \geq 1}{i\in[n]}$ are the
inequalities of the original system $\sys_0$ satisfied at equality by
$x^*$. In order to find an integral generating set for $K(x^*) \cap \Z^n$ we
need the following result.

\begin{lemma}
\label{th:Cnk-sys}
Let $x \in \R^n, b \in \Z^n$ be two vectors such that $\zerovec
\leq x < \onevec$ and $\C{n}{k} x = b$, with $\gcd(n,k)=1$. Then there
exists $r \in \tabulatedset{0, 1, \ldots, k-1}$ such that $x=
\frac{r}{k} \onevec$ and $b= r \onevec$.
\end{lemma}

\begin{proof}
Observe that $x^*$ is the solution of the linear system 
$$
C^i x = \sum_{j=i}^{i + k-1}x_j = b_i, \quad \forall 1 \leq i \leq n.
$$ 
Subtracting each equation from the previous one, and the first from
the last, we obtain $x_i - x_{i + k} = b_i - b_{i + 1}$,
for $1 \leq i \leq n$. Moreover, from $\zerovec \leq x < \onevec$ it
follows $-1 < x_i - x_{i + k} < 1$ and, since $b$ is integral, we
must have $x_i - x_{i + k} = b_i - b_{i + 1} = 0$. As a
consequence, $b = r \onevec$ for some $r \in \Z$.

On the other hand, as $\C{n}{k} \onevec = k \onevec$, we can write $b
= \C{n}{k} (\frac{r}{k} \onevec)$ and substitute this expression in
the original system. Since $\gcd(n,k)=1$, the matrix $\C{n}{k}$ is
invertible and we obtain $x= \frac{r}{k} \onevec$. Finally, from
$\zerovec \leq x < \onevec$ we have $r \in \tabulatedset{0, 1, \ldots,
  k-1}$.
\end{proof}

With this result we can compute an integral generating set for $K(x^*)$.

\begin{theorem}\label{1sobrek}
Let $\C{n}{k}$ be a circulant matrix such that $\gcd(n,k)=1$ and
consider the vertex $x^*=\frac{1}{k} \onevec$ of $\Q{\C{n}{k}}$. Then
an integral generating set for $K(x^*) \cap \Z^n$ is given by $H(x^*)=
\tabulatedset{ C^1, C^2, \ldots, C^n, \onevec}.$
\end{theorem}

\begin{proof}
Let $b \in K(x^*) \cap \Z^n$. Then there exists a non negative
vector $\hat{d} \in \R_{+}^n$ such that $b^T = \hat{d}^T \C{n}{k}$. 
Moreover, as the (unordered) sets of row and column vectors
of a circulant matrix are equal, this is equivalent to saying that
there exists $d \in \R_{+}^n$ with $\C{n}{k}d = b$. Let
$x= d - \floor{d}$. We have $\zerovec \leq x < \onevec$ and 
$$
\C{n}{k}(x + \floor{d}) = b \quad \Leftrightarrow \quad
\C{n}{k} x = b - \C{n}{k} \floor{d}.
$$

Since $b - \C{n}{k} \floor{d}$ is integral, applying
Lemma~\ref{th:Cnk-sys} yields $\C{n}{k} x= r \onevec$,
for some $r \in \tabulatedset{0, 1, \ldots, k-1}$. Hence,
$$
b = r \onevec + \C{n}{k} \floor{d}
$$ 
and thus $b$ can be written as a linear combination
of the elements of $H(x^*)$ where all coefficients are
non negative integers.
\end{proof}

When applying step~2.2. of the procedure described at the beginning of this section,
the vectors $C^i$ yield the inequalities $(C^i)^T x \geq 1$ from $\sys_0$,
while for the last vector we obtain the rank constraint of
$\Qi{\C{n}{k}}$:

\begin{corollary}
If $\gcd(n,k)=1$, then the inequality 
$\onevec^T x \geq \ceil{\onevec^T x^*} = \ceil{\frac{n}{k}}$ 
is valid for $\Qp{\C{n}{k}}$.
\end{corollary}

On the other hand, given a vertex $x^*$ corresponding to a circulant
minor $\C{n'}{k'}$, the task of finding an integral generating set for
$K(x^*) \cap \Z^n$ turns out to be more complicated. We present here
a partial result which is however sufficient for deriving a new class
of facet-defining inequalities for $\Qi{\C{n}{k}}$.

Consider a circulant minor $\C{n'}{k'} \approx \C{n}{k}/N$ of
$\C{n}{k}$, and let $x^*$ be the corresponding vertex of
$\Q{\C{n}{k}}$, defined as in Lemma~\ref{th:vertices}(iii). From
Lemma~2.1 and Lemma~2.4 in \cite{Nes} it follows that $\C{n'}{k'}$ can be
obtained from  $\C{n}{k}$ by deleting each column $j$ with $j \in N$
and each row $i$ with $i+1 \in N$.
Hence, $\card{C^i \setminus N}= k'$ holds for $i+1 \not\in N$ and 
$x^*$ satisfies the following inequalities with equality:
\begin{align}
&(C^i)^T x \geq 1, \mbox{ for all } i \mbox{ such that } i+1 \not\in N; \label{eq:cnk}\\ 
& (e^j)^T x \geq 0, \mbox{ for all } j \in N. \label{eq:nonneg}
\end{align}
Observe that there might be other inequalities from $\sys_0$ satisfied tightly
by $x^*$. In the following we denote by $K$ the subcone of $K(x^*)$
spanned by the normal vectors of the left-hand sides of \eqref{eq:cnk}
and \eqref{eq:nonneg}.
\begin{theorem}
\label{th:minors-gs}
Let $x^*$ be a vertex of $\Q{\C{n}{k}}$ associated with a minor $\C{n}{k}/ N \approx \C{n'}{k'}$,
and let $W=\setof{i\in N}{i-k-1 \in N}$. An integral generating set for $K \cap \Z^n$ is given by
$$
\setof{C^i}{i+1 \not\in N} \cup \setof{e^j}{j \in N} \cup
\setof{r\onevec + \sum_{j \in W} e^j}{1 \leq r \leq k'-1}.
$$
\end{theorem}

\begin{proof}
Let $A$ be the square coefficient matrix of the system
\eqref{eq:cnk}-\eqref{eq:nonneg}. Reordering the columns, we may
assume that $A$ has the following block form:
$$
A= \left(
\begin{array}{c|c|c}
B^1 & B^2 & C \\
\hline
I_1 & \OM & \OM \\
\hline
\OM & I_2 & \OM
\end{array}
\right),
$$ 
where the first $n_3:= \card{W}$ columns correspond to indices in
$W$, the next $n_2:= \card{N - W}$ columns correspond to indices in
$N-W$ and the last $n'$ columns correspond to indices not in $N$.
Moreover, $C:=\C{n'}{k'}$, and $I_1, I_2, \OM$ are identity 
and zero matrices of the appropriate sizes.

Let $b \in K \cap \Z^n$. Then there exist $d \in \R_{+}^{n'}$,
$f \in \R_{+}^{n_3}$, and $g \in \R_{+}^{n_2}$ with $b^T=(d^T,
f^T, g^T) A$. Now consider the vectors $x:= d - \floor{d}$,
$y:= f - \floor{f}$, $z:= g - \floor{g}$, and define:
\begin{equation}
\label{eq:btilde}
\tilde{b}^T:= b^T - \left(\floor{d^T}, \floor{f^T}, \floor{g^T}\right) A.
\end{equation}
Observe that $\tilde{b} \in \Z^n$ and $\tilde{b}^T=(x^TB^1 + y^T,
x^TB^2 + z^T, x^TC)$. From Lemma~\ref{th:Cnk-sys} it follows
that $x^TC = r \onevec^T$ and
$x= \frac{r}{k'} \onevec$, for some $r \in \tabulatedset{0, \ldots,
k'-1}$. Moreover, if $r=0$ then we must have $\tilde{b}=0$
and from \eqref{eq:btilde} we conclude that $b^T$ is an
integral conic combination of the row vectors of $A$,
which finishes the proof.

Now assume $r \in \tabulatedset{1, \ldots, k'-1}$. For $\ell \in
\tabulatedset{1, 2}$, we have $x^TB^\ell = \frac{r}{k'} {\onevec}^T
B^\ell$. From the proof of Theorem~3 in \cite{BNTDAM13}, it follows 
that each column of $B^1$ has support equal to $k' + 1$ and
each column of $B^2$ has support equal to $k'$. Thus,
$$ 
\tilde{b}^T=\left(\frac{r(k'+1)}{k'} \onevec^T + y^T, r \onevec^T + z^T, r \onevec^T\right).
$$ 

Finally, since $\tilde{b}^T \in \Z^n$, $\zerovec \leq y <
\onevec$, and $\zerovec \leq z < \onevec$, we must have
$y=(1-\frac{r}{k'}) \onevec$, $z= \zerovec$, and hence
$\tilde{b}^T=((r+1) \onevec^T, r \onevec^T, r \onevec^T) = r \onevec^T
+ \sum_{j \in W} (e^j)^T$. Then the statement of the theorem follows
from \eqref{eq:btilde}.
\end{proof}

Vectors in the first two sets of last theorem give rise to boolean inequalities
after applying step~2.2. of the procedure from the beginning of this section. 
For the third set, we have $\ceil{r \onevec^T x^* + \sum_{j \in W} (e^j)^T x^*} = 
\ceil{\frac{rn'}{k'}}$, and hence we obtain:

\begin{corollary}
Let $N\subset [n]$ be such that $\C{n}{k}/N \approx \C{n'}{k'}$ and $W=\{i\in N : i-k-1 \in N \}$.
The inequalities 
\begin{equation}
\label{eq:r-minors}
\sum_{i \in W} (r+1) x_i + \sum_{i \not\in W} r x_i \geq
\ceil{\frac{rn'}{k'}},
\end{equation}
with $r \in \tabulatedset{1, \ldots, k'-1}$, are valid for $\Qp{\C{n}{k}}$.
\end{corollary}

For $r=1$, these inequalities are the minor inequalities described in
\cite{BNTDAM13}.  Accordingly, we have called \eqref{eq:r-minors} as
\emph{generalized $r$-minor inequalities}.  In some cases, generalized $r$-minor 
inequalities with $r> 1$ can be obtained from the addition of (classical) minor
inequalities and the rank constraint, and are thus redundant for
$\Qp{\C{n}{k}}$. 
However, this is not true in general. For instance, consider $W:=
\setof{6+5k}{0 \leq k \leq 10}$ and $N:= W \cup
\tabulatedset{1}\subset [59]$.  One can verify that $\C{59}{4}/N
\approx\C{47}{3}$ and the corresponding inequality \eqref{eq:r-minors}
for $r=2$ has the form $\sum_{i \in W} 3x_i + \sum_{i \not\in W} 2x_i
\geq 32$. Moreover, it can be shown that this inequality defines a facet of
$\Qi{\C{59}{4}}$. As a consequence we have the following result.

\begin{theorem}
There are circulant matrices $\C{n}{k}$ for which minor inequalities,
boolean facets, and the rank constraint are not enough to describe
$\Qp{\C{n}{k}}$.
\end{theorem}

\section{Separation algorithms for generalized minor inequalities}
\label{sec:separation}

A polynomial 
time algorithm to separate 1-minor inequalities associated with particular classes of circulant minors
has been proposed by S. Bianchi et al.~in \cite{BNTDAM13}. In this section we extend some of their results 
to the case of generalized minor inequalities. More precisely, we address the separation problem 
for $r$-minor inequalities corresponding to circulant minors having parameters $d=n_1=1$.

Following the ideas presented in \cite{BNTDAM13}, let us first prove a technical lemma 
that will be required for our separation procedure. 

\begin{lemma}\label{alfa}
Let $d, n_1=1, n_2$, and $n_3$ be the parameters associated with a circulant minor of $C^k_n$ such that 
$n_3=p\, (\mathrm{mod} \,(k-d))$ 
with  $1\leq p< k-d$. 
Then
$$
\left\lceil \frac{rn'}{k'}\right\rceil=
\left\lceil r\ \frac{n-d(n_2+n_3)}{k-d}\right\rceil= \left[r\frac{n}{k}+\left(\left\lceil \frac{rp}{k-d}\right\rceil-\frac{rp}{k-d}\right)\right]+\frac{r}{k(k-d)}dn_3. $$  
\end{lemma}

\begin{proof} 
Let $s$ be the nonnegative integer such that $n_3=s(k-d)+p$. 
Since $n= n_1 n =k(n_2+n_3)+n_3$ we have that 
$$\left\lceil r\ \frac{n-d(n_2+n_3)}{k-d}\right\rceil= \left\lceil r\ \frac{(k-d)(n_2+n_3)+n_3}{k-d}\right\rceil= r\left(n_2+n_3\right)+\left\lceil \frac{rn_3}{k-d}\right\rceil.$$

From $s=\frac{n_3- p}{k-d}$ and $n_2+n_3=\frac{n-n_3}{k}$ it follows that 
\begin{align*}
r\left(n_2+n_3\right)+\left\lceil \frac{rn_3}{k-d}\right\rceil &= r\ \frac{n-n_3}{k} + r\ \frac{n_3-p}{k-d}+\left\lceil \frac{rp}{k-d}\right\rceil\\ 
& =\left[r\frac{n}{k}+\left(\left\lceil \frac{rp}{k-d}\right\rceil-\frac{rp}{k-d}\right)\right]+\frac{r}{k(k-d)}dn_3
\end{align*}
and the proof is complete.
\end{proof}

Observe that if $n_3$ is a multiple of $k-d$, then the corresponding $r$-minor inequality
is redundant for $\Q{\C{n}{k}}$, as the value $\frac{rn'}{k'}$ on the right-hand side is integer.
Otherwise, if $W\subset [n]$ defines a minor of $C_n^k$ with parameters $d, n_1=1, n_2$, and $n_3$, where $n_3=p\, (\mathrm{mod} \,(k-d))$ 
and  $1\leq p< k-d$, the previous lemma implies that
the corresponding $r$-minor inequality can be written as 
$$\sum_{i\in W} x_i+r\sum_{i=1}^n x_i  \geq  \alpha(d,p) +\beta(d) \left|W\right|$$ 
where 
$$\alpha(d,p)= r\frac{n}{k}+\left(\left\lceil \frac{rp}{k-d}\right\rceil-\frac{rp}{k-d}\right),\quad \beta(d)=\frac{r}{k(k-d)},$$
or, equivalently,
\begin{equation}\label{faceta}
\sum_{i\in W} (x_i - \beta(d))  \geq   \alpha(d,p)- r\sum_{i=1}^n x_i.
\end{equation}

Given $C^k_n$ and two integer numbers $d,p$ with $1\leq d\leq k-2$ and $1\leq p <k-d$,  
we define the function $c^d$ on $\R$ by $c^d(t):= t - \beta(d)$ and the function $L^{d,p}$ on $\R^n$ by $L^{d,p}(x):= \alpha(d,p)- r\sum_{i=1}^n x_i$.  

Then, inequality (\ref{faceta}) can be written as

\begin{equation}\label{faceta2}
\sum_{i\in W} c^d (x_i) \geq   L^{d,p} (x).
\end{equation}

Following the same notation introduced in \cite{BNTDAM13}, let $\mathcal{W}(d,p)$ be the family of sets $W\subset [n]$ defining minors with parameters $d, n_1=1, n_2, n_3=p\, (\mathrm{mod} \, (k-d))$. 
We are interested in the separation of generalized $r$-minor inequalities corresponding to
sets $W \in \mathcal{W}(1,p)$.

To this end, given $n,k$ 
let $K_n^k(p)=\left(V, A\right)$ be the digraph with set of nodes 
$$V=\bigcup\limits_{j\in [k-1]}V^j \cup \{t\}$$
where $V^j=\{v^j_i :i\in [n]\}$ 
and set of arcs defined as follows: first consider in $A$ the arcs  
\begin{itemize}
	\item $(v_1^1,v_l^{2})$ for all $l$ such that $k+2\leq l\leq n$ and $l=2$ (mod $k$),
\end{itemize}
then consider in a recursive way for $j\in[k-1]$:
\begin{itemize}
	\item for each $(v,v_i^{j})\in A$, add $(v_i^j,v_l^{j+1})$ whenever $l$ is such that $i+k+1\leq l\leq n$ and $l-i=1$ (mod $k$),
\end{itemize}
and finally,
\begin{itemize}
  \item for each $(v,v_i^{p})\in A$, add $(v_i^p,t)$ whenever $i$ is such that $i\leq n-k$ and $n-i=0$ (mod $k$).
\end{itemize}

Note that, by construction, $K_n^k(p)$ is acyclic. In Figure \ref{fig:caminosep} we sketch the digraph $K_{59}^4(2)$ where only the arcs corresponding to a particular $v_1^1t$-path are drawn.

\begin{figure}[h]
	\centering
		\includegraphics[width=.32\textwidth]{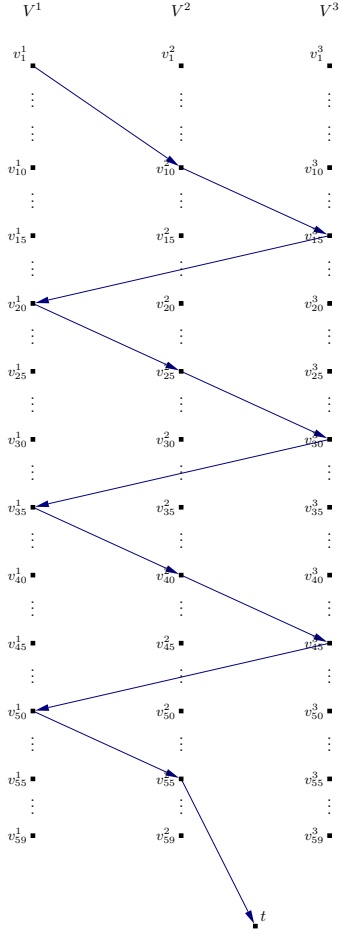}
	\caption{A $v_1^1t$-path in the digraph $K_{59}^4(2)$.}
	\label{fig:caminosep}
\end{figure}

For the proof of the next lemma, we need the following observation.
\begin{remark}\label{condW} 
Let $W=\{w_i: i\in [n_3]\}$ with 
$1\leq w_1< \dots< w_{n_3}\leq n$, then $W$ defines a circulant minor with parameters $d=n_1=1$ if and only if $w_{i+1}-w_{i}=1$ (mod $k$) and $w_{i+1}-w_{i}\geq k+1$, for all $i\in [n_3]$.
\end{remark}

Now we can state our result.

\begin{lemma}
There is a one-to-one correspondence between $v_1^1 t$-paths in $K_n^k(p)$ and subsets $W\in \mathcal{W}(1,p)$ with $1\in W$.
\end{lemma}

\begin{proof}
 
Let $W\in \mathcal{W}(1,p)$ and assume that $W=\{i_j: j\in [n_3]\}\subset [n]$ with $1=i_1<i_2<\ldots<i_{n_3}\leq n$. 

Then, by Remark \ref{condW}, $i_{j+1}- i_j= 1$ (mod $k$) and $i_{j+1}- i_j\geq k+1$ for all $j\in [n_3]$. 
Thus, the set of nodes $P \cup \left\{t\right\}$ with
$$P:=\left\{v_{i_j}^\ell\,:\, i_j\in W ,\,\, \ell=j\, (\mathrm{mod} \, (k-1)) \right\}$$
induces a $v_1^1 t$-path in $K_n^k(p)$.

Conversely, let $P$ be a $v_1^1 t$-path in $K_n^k(p)$. By construction, there exists a 
nonnegative integer $s$ such that 
$|V(P)\cap V^j|= s+1$ for all $j\in \{1, \ldots, p\}$ and $|V(P)\cap V^j|= s$ for $j \in \{p+1, \ldots, k-1\}$. 
Then, $|V(P)-\{t\}|=s(k-1)+p$.

Now, if we define 
$$W= \setof{i \in [n]}{v_i^j \in V(P) \mbox{ for some } j \in [k-1]}$$ 
we have $|W|=s(k-1)+p$ and from Remark \ref{condW} it follows that $W\in \mathcal{W}(1,p)$. 
\end{proof}

\begin{theorem}\label{sep}
Given $C^k_n$ and $r \in \{ 1, \ldots, k-1\}$, the separation problem for $r$-minor inequalities corresponding to minors with parameters $d=n_1=1, n2,$ and $n_3=p\, (\mathrm{mod} \,(k-1))$ can be polynomially reduced to at most $n$ shortest path problems in an acyclic digraph.
\end{theorem}

\begin{proof} Let $\hat{x} \in \mathbb{R}^n$. We will show that the problem of deciding if, given $j\in [n]$, there exists $W\in \mathcal{W}(1,p)$ with $j\in W$ and such that $\hat{x}$ violates the inequality (\ref{faceta2}) can be reduced to a shortest path problem. W.l.o.g we set $j=1$.  

Consider the digraph $K_n^k(p)$ and associate the cost $c^1(\hat{x_i})$ with every arc $(v_l^j, v_i^{j+1})\in A$, and the cost $c^1(\hat{x_1})$ with every arc $(v_l^p, t)\in A$.

Clearly, if $W$ is the subset of $[n]$ corresponding to a $v_1^1 t$-path $P$ in $K_n^k(p)$, the length of $P$ is equal to $\sum_{i\in W} c^1(\hat{x_i})$.

Then, there exists $W\in \mathcal{W}(1,p)$ with $1\in W$ and such that $\hat{x}$ violates the inequality \eqref{faceta2}, if and only if the length of the shortest $v_1^1 t$-path in $K_n^k(p)$ is less than 
$L^{1,p} (\hat{x})$. 
\end{proof}

Since $K_n^k(p)$ is acyclic, computing each shortest path from the last theorem can be accomplished in $O(\card{A})$ time (see, e.g., \cite[Theorem 2.18]{CookEtAl98}). Moreover, from the definition 
of $K_n^k(p)$ it follows that each node has outdegree
of order $O(\frac{n}{k})$ and that the graph contains $O(nk)$ nodes. Hence, each shortest
path computation requires $O(n^2)$ time, and the separation problem of $r$-minor inequalities can be solved in $O(n^3)$ time for fixed $p,r \in \{1, \ldots, k-1 \}$.
Repeating this procedure for each possible value of $p$ and $r$, we obtain the following result.

\begin{theorem}\label{sep2}
For a fixed $k$, the separation problem for $r$-minor inequalities corresponding to minors of $C_n^k$ with parameters $d=n_1=1$ can be solved in polynomial time.
\end{theorem}

\section{Conclusions}
\label{sec:conclusions}

We have presented a new class of valid inequalities for $\Qi{\C{n}{k}}$
whose Chvátal rank is at most one. These inequalities strictly generalize the class
of minor inequalities described in \cite{BNTDAM13}. Moreover, some
of these inequalities give rise to new facets of $\Qi{\C{n}{k}}$, as shown by
the example at the end of Section~\ref{sec:clausura}. Hence, 
despite of the results obtained for $\Qi{C^3_n}$ in \cite{Mah} and $\Qi{C_{2k}^k}, \Qi{C_{3k}^k}$
in \cite{BNTMMOR13}, (classic) minor inequalities, together with boolean 
facets and the rank constraint are not sufficient to provide a complete linear 
description of $\Qi{\C{n}{k}}$ in general. 

An apparently weaker problem consists in finding a complete linear description for the
first Chvátal closure of $\Q{\C{n}{k}}$.  However, as far as we are aware from
literature, no circulant matrix $\C{n}{k}$ is known for which the Chvátal rank of
$\Q{\C{n}{k}}$ is strictly larger than one. To complete the characterization
of $\Qp{\C{n}{k}}$ following the path presented here, more work is still 
needed in order to fully characterize integral generating sets for the cones 
$K(x^*)$ associated with the fractional vertices defined in Lemma~\ref{th:vertices}(iii).
All computational experiments we have conducted so far support the conjecture
that generalized minor inequalities, together with boolean facets and the rank constraint
are sufficient for describing $\Qp{\C{n}{k}}$.

Besides of the search for complete linear descriptions of $\Qp{\C{n}{k}}$ and $\Qi{\C{n}{k}}$, 
one line of future research could be the study of necessary and sufficient conditions for a 
generalized minor
inequality to define a facet of the integer polytope, similar to the conditions presented in \cite{BNTDAM13}
for classic minor inequalities. Moreover, the separation problem for generalized minor
inequalities corresponding to minors with parameters $d>1$ or $n_1>1$ is an issue
that requires further investigation. 

Another possible line of research involves establishing analogies between set
packing and set covering polyhedra. As mentioned in the introduction, a
complete linear description for set packing polytope $\Pint{\C{n}{k}}$ related 
to circulant matrices (in fact, to the more general class of \emph{circular} matrices) has 
been reported in \cite{EisenbrandEtAl05}. The authors show that
this polytope is described by nonnegativity constraints, clique inequalities
and so-termed \emph{clique family} inequalities. The first two families
can be regarded as counterparts of boolean facets for the set covering
polyhedron. Determining whether there is also an analogous class to clique family inequalities
in the case of $\Qi{\C{n}{k}}$, 
and how this class is related to the class of generalized minor inequalities
could shed more light
on the structure of this polyhedron. For instance, similarly to generalized minor inequalities, clique family inequalities have only two different coefficients, which are consecutive integers.
First steps in this direction have been undertaken in \cite{rowfamily,genrowfamily}.


\begin{thebibliography}{10}\label{bibliography}


\bibitem{Nes} Aguilera, N., \emph{On packing and covering polyhedra of consecutive ones circulant clutters}, Discrete Applied Mathematics
(2009), 1343--1356.
 
\bibitem{circu} Argiroffo, G. and S. Bianchi, \emph{On the set covering polyhedron of circulant matrices}, Discrete Optimization \textbf{6} (2009), 162--173.

\bibitem{rowfamily} Argiroffo, G. and S. Bianchi, \emph{Row family inequalities for the set covering polyhedron}, Electronic Notes in Discrete Mathematics \textbf{36} (2010), 1169--1176.

\bibitem{genrowfamily} Argiroffo, G. and A. Wagler, \emph{Generalized row family inequalities for the set covering polyhedron}.  In: Proceedings of the 10th Cologne-Twente Workshop 2011. (2011), 60--63.

\bibitem{BalasNg89} Balas E. and S.M. Ng, \emph{On the set covering polytope: I. All the facets with coefficients in \{0, 1, 2\}}, Mathematical Programming \textbf{43} (1989), 57--69.

\bibitem{libro} Bertsimas, D. and R. Weismantel, \emph{Optimization Over Integers}, Dynamic Ideas, Belmont, Massachusetts (2005).

\bibitem{BNTDAM13} Bianchi, S., G. Nasini and P. Tolomei, \emph{Some advances on the set covering polyhedron of circulant matrices}, Discrete Applied Mathematics (2013), in press. \url{http://dx.doi.org/10.1016/j.dam.2013.10.005}.

\bibitem{BNTMMOR13} Bianchi, S., G. Nasini and P. Tolomei, \emph{The Minor inequalities in the description of the Set Covering Polyhedron of Circulant Matrices}, Mathematical Methods of Operations Research (2013). DOI 10.1007/s00186-013-0453-6.

\bibitem{Mah} Bouchakour, M., T.M. Contenza, C.W. Lee and A.R. Mahjoub, \emph{On the dominating set polytope}, European Journal of Combinatorics \textbf{29 (3)} (2008), 652--661.

\bibitem{CookEtAl98} Cook, W.J., W.H. Cunningham, W.R. Pulleyblank and A. Schrijver, \emph{Combinatorial Optimization}, Wiley-Interscience, New York, (1998).

\bibitem{CornuejolsSassano89} Cornu{\'e}jols G. and A. Sassano, \emph{On the 0,1 facets of the set covering polytope}, Mathematical Programming \textbf{43} (1989), 45--55.

\bibitem{CN} Cornu\'{e}jols, G. and  B. Novick, \emph{Ideal $0-1$ Matrices},
Journal of Combinatorial Theory B \textbf{60} (1994), 145--157.

\bibitem{EisenbrandEtAl05} Eisenbrand, F., G. Oriolo, G. Stauffer and P. Ventura, \emph{The Stable Set Polytope of Quasi-Line Graphs}, Combinatorica \textbf{28 (1)} (2008), 45--67.

\bibitem{NobiliSassano89} Nobili P. and A. Sassano, \emph{Facets and lifting procedures for the set covering polytope}, Mathematical Programming \textbf{45} (1989), 111--137.

\bibitem{Sa} Sassano, A., \emph{On the facial structure of the set covering polytope}, Mathematical Programming \textbf{44} (1989), 181--202.

\bibitem{TolomeiTorres13} Tolomei, P. and L.~M.~Torres, \emph{On the first Chv{\'a}tal closure of the set covering polyhedron related to circulant matrices}, Electronic Notes in Discrete Mathematics \textbf{44} (2013), 377--383.

\end{thebibliography}
\end{document}